\documentclass[a4paper,12pt]{article}%
\usepackage[T1]{fontenc}
\usepackage[english]{babel}
\usepackage{amsmath,amsthm,amssymb,}
\usepackage{esint}
\usepackage{tikz}
\usepackage[caption=false]{subfig}
\usepackage{url}
\usepackage[left=2.5cm,right=2.5cm,top=2.5cm,bottom=2.5cm]{geometry}
\usepackage{fancyhdr}
\makeatother

\newtheorem{theorem}{Theorem}

\newtheorem{lemma}[theorem]{Lemma}

\theoremstyle{definition}
\newtheorem{remark}[theorem]{Remark}
\numberwithin{figure}{section}
\numberwithin{equation}{section}
\numberwithin{theorem}{section}

\newcommand{\R}{\mathbb{R}}
\newcommand{\tri}{\mathcal{T}}
\newcommand{\E}{\mathcal{E}}
\newcommand{\N}{\mathcal{N}}
\newcommand{\dx}{\,dx}
\newcommand{\ds}{\,ds}
\newcommand{\C}{C(\Omega)}
\newcommand{\Cnull}{C_0(\Omega)}
\newcommand{\gradnc}{\nabla_{\operatorname{NC}}}
\newcommand{\divnc}{\operatorname{div}_{\operatorname{NC}}}
\newcommand{\Inc}{I_{\operatorname{NC}}}
\newcommand{\oscf}{\operatorname{osc}(f,\tri)}
\newcommand{\CR}{\operatorname{CR}^1_0(\tri)}
\newcommand{\Vcr}{V_{\operatorname{CR}}(\tri)}
\newcommand{\ucr}{u_{\operatorname{CR}}}
\newcommand{\pcr}{p_{\operatorname{CR}}}
\newcommand{\vcr}{v_{\operatorname{CR}}}
\newcommand{\wcr}{w_{\operatorname{CR}}}
\newcommand{\qcr}{q_{\operatorname{CR}}}
\newcommand{\Vm}{V_{\operatorname{MINI}}(\tri)}
\newcommand{\B}{\mathcal{B}}
\newcommand{\um}{u_{\operatorname{MINI}}}
\newcommand{\pmini}{p_{\operatorname{MINI}}}
\newcommand{\vm}{v_{\operatorname{MINI}}}
\newcommand{\qm}{q_{\operatorname{MINI}}}
\newcommand{\ulin}{u_{\operatorname{lin}}}
\newcommand{\ub}{u_{\operatorname{b}}}

\newcommand{\ubr}{u_{\operatorname{BR}}}
\newcommand{\vbr}{v_{\operatorname{BR}}}
\newcommand{\pbr}{p_{\operatorname{BR}}}
\newcommand{\qbr}{q_{\operatorname{BR}}}
\newcommand{\Vbr}{V_{\operatorname{BR}}(\tri)}
\newcommand{\up}{u_{\operatorname{P2}}}
\newcommand{\vp}{v_{\operatorname{P2}}}
\newcommand{\pp}{p_{\operatorname{P2}}}
\newcommand{\qp}{q_{\operatorname{P2}}}
\newcommand{\Vp}{V_{\operatorname{P2}}(\tri)}

\newcommand{\tr}{\operatorname{tr}}

\newcommand{\PS}{{\operatorname{PS}}}
\newcommand{\dev}{\operatorname{dev}}
\newcommand{\ddiv}{\operatorname{div}}

\begin{document}


\title{Comparison results for the Stokes equations}

\author{C.~Carstensen\footnote{Institut f\"ur Mathematik, 
            Humboldt-Universit\"at zu Berlin, 
            Unter den Linden 6, D-10099 Berlin,Germany}~\footnote{Department of Computational Science and Engineering,
            Yonsei University, Seoul, Korea}~
            K.~K\"ohler\footnotemark[1]~
            D.~Peterseim\footnote{Institut f\"ur Numerische Simulation, Universit\"at Bonn,
            Wegelerstr.\ 6, 53111 Bonn, Germany}~\footnote{Corresponding author. Tel.: +49 228 73-2058}~
            M.~Schedensack\footnotemark[1]}
\maketitle
\begin{abstract}
This paper enfolds a medius analysis for the Stokes equations and
compares different finite element methods (FEMs). 
A first result is a best approximation result for a $P_1$ non-conforming FEM. 
The main comparison result is that the error of the $P_2 P_0$-FEM
is a lower bound to the error of
the Bernardi-Raugel (or reduced $P_2 P_0$) FEM, 
which is a lower bound to the error of
the $P_1$ non-conforming FEM,
and this is a lower bound to the
error of the MINI-FEM.
The paper discusses the converse direction, as well as other methods
such as the discontinuous Galerkin and pseudostress FEMs.

Furthermore this paper provides 
counterexamples for equivalent convergence when different pressure 
approximations are considered.
The mathematical arguments are various conforming companions 
as well as the discrete inf-sup condition. 

\end{abstract}



\section{Introduction}

\begin{figure}[h!]
  \begin{center}
\subfloat[\label{f:FEMsMini}MINI-FEM]{
 \begin{tikzpicture}[x=7.5mm,y=12.5mm]
\foreach \a in {0,2.5,5.5}
{ \draw (\a,0)--(1+\a,1)--(2+\a,0)--cycle; }
\foreach \a/\b in
{0/0,1/1,2/0,2.5/0,3.5/1,4.5/0,1/0.33333,3.5/0.333333,5.5/0,7.5/0,6.5/1}
{ \fill (\a,\b) circle (2pt); }
 \end{tikzpicture}
}
\hspace{1cm}
\subfloat[\label{f:FEMsCR}CR-NCFEM]{
 \begin{tikzpicture}[x=7.5mm,y=12.5mm]
\foreach \a in {6,8.5,11.5}
{ \draw (\a,0)--(1+\a,1)--(2+\a,0)--cycle; }
\foreach \a/\b in {6.5/0.5,7.5/0.5,7/0,12.5/0.33333,9/0.5,10/0.5,9.5/0}
{ \fill (\a,\b) circle (2pt); }%
%
 \end{tikzpicture}
}
\end{center}
	\caption{\label{f:FEMs}MINI-FEM and CR-NCFEM for the Stokes equations.}
\end{figure}
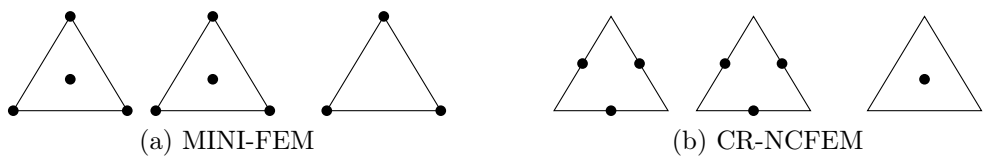

\begin{figure}
  \begin{center}
\subfloat[\label{f:FEMsP2P0}$P_2 P_0$-FEM]{
 \begin{tikzpicture}[x=7.5mm,y=12.5mm]
\foreach \a in {0,2.55,5.5}
{ \draw (\a,0)--(1+\a,1)--(2+\a,0)--cycle; }
\foreach \a/\b in {0/0,1/1,2/0,0.5/0.5,1./0,1.5/0.5}
{ \fill (\a,\b) circle (2pt); }
\foreach \a/\b in {2.55/0,3.55/1,4.55/0,3.55/0,4.05/0.5,3.05/0.5}
{ \fill (\a,\b) circle (2pt); }
\foreach \a/\b in {6.5/0.333333}
{ \fill (\a,\b) circle (2pt); }
\phantom{\draw (0,-0.5)--(1,-0.5);}
 \end{tikzpicture}
}
\hspace{1cm}
\subfloat[\label{f:FEMsBR}BR-FEM]{
 \begin{tikzpicture}[x=7.5mm,y=12.5mm]
\foreach \a in {11.8,14.55,17.5}
{ \draw (\a,0)--(1+\a,1)--(2+\a,0)--cycle; }
\foreach \a/\b in {11.8/0,13.8/0,12.8/1,14.55/0,15.55/1,16.55/0,18.5/0.33333}
{ \fill (\a,\b) circle (2pt); }
\foreach \a/\b in {13/0,13.5/0.5,12.5/0.5}
{\draw[->] (12.3,0.5)->(11.55,0.75);}
{\draw[->] (12.8,0)->(12.8,-0.5);}
{\draw[->] (13.3,0.5)->(14.05,0.75);}
\foreach \a/\b in {15.55/0,16.05/0.5,15.05/0.5}
{\draw[->] (15.05,0.5)->(14.3,0.75);}
{\draw[->] (15.55,0)->(15.55,-0.5);}
{\draw[->] (16.05,0.5)->(16.8,0.75);}
\end{tikzpicture}
}
\end{center}
	\caption{\label{f:FEMP2s}$P_2P_0$-FEM and BR-FEM for the Stokes 
equation.}
\end{figure}
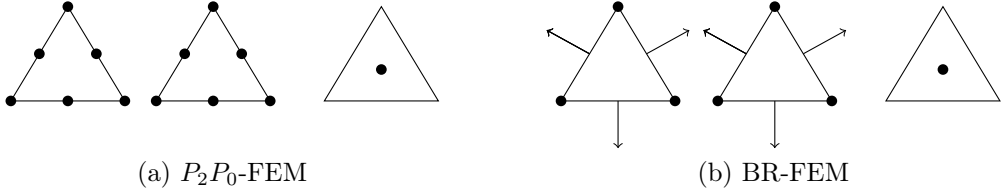

Given some external force $f\in L^2(\Omega;\mathbb{R}^2)$ in some polygonal
Lipschitz
domain $\Omega$, the Stokes equations seek the velocity field $u\in
H^1_0(\Omega;\mathbb{R}^2):=\{u\in
H^1(\Omega;\mathbb{R}^2)\;\vert\;u\vert_{\partial\Omega}=0\text{ in the sense of
traces}\}$ and the pressure distribution $p\in
L^2_0(\Omega):=\{q\in L^2(\Omega)\;\vert\;
\int_\Omega q\dx\linebreak =0\}$
with
\begin{align}\label{e:StokesProblem}
-\Delta u+\nabla p =f 
\quad\text{and}\quad
\operatorname{div}u=0
\quad\text{in }\Omega.
\end{align}
This paper compares several standard mixed finite element methods for the
numerical approximation of the unknown solution 
pair $(u,p)\in H^1_0(\Omega;\mathbb{R}^2)\times L^2_0(\Omega)$ in terms of
accuracy. 
Comparison results for the Poisson model problem 
of \cite{Braess09,CPS12} give rise to the conjecture that 
first-order finite element methods (FEMs) for the 
Stokes problem are comparable in the sense that their errors on the same mesh are 
equivalent up to multiplicative constants, which are independent of the local 
mesh-size.
The aim of this paper is to investigate the comparability of FEMs that are 
conceptually very different. The considered FEMs are MINI-FEM, CR-NCFEM, 
$P_2P_0$-FEM and BR-FEM (cf.\ Figures~\ref{f:FEMs}--\ref{f:FEMP2s}). 
Since they use different continuous and discontinuous 
approximations of the velocity and/or the pressure,
the approximation properties of the ansatz spaces do not allow 
for equivalence but only for a comparison in one direction.

The constraint $\operatorname{div}u=0$ excludes standard 
piecewise affine FEMs
based on continuous piecewise affine approximations
of the velocity components (see, e.g., \cite{BrennerScott08}).
The MINI-FEM from Figure~\ref{f:FEMsMini} (see Section \ref{ss:defFEMs} for a 
precise definition) is a conforming method which fulfils the constraint $\ddiv 
u=0$ in a weak sense only. It is based on a piecewise affine approximation of 
the velocity with an additional bubble function on each triangle for each 
component of the velocity. 

The $P_1$ non-conforming FEM, CR-NCFEM, 
from Figure~\ref{f:FEMsCR} (see Section~\ref{ss:defFEMs} for 
the precise definition),
however, fulfils this 
constraint element-wise.
While for the MINI-FEM the best approximation result
\begin{align*}
 \|\nabla(u&-\um)\|+\|p-\pmini\|\\
& \lesssim \min_{\vm\in\Vm}\|\nabla(u-\vm)\| 
   + \min_{\qm\in P_1(\tri)\cap\C\cap L^2_0(\Omega)} \|p-\qm\|
\end{align*}
is a direct consequence of the conformity and stability,
this paper proves the best approximation result
\begin{align*}
 \|\gradnc&(u-\ucr)\|+\|p-\pcr\|\\
 &\lesssim \min_{\vcr\in\Vcr}\|\gradnc(u-\vcr)\|
      +\min_{\qcr\in P_0(\tri)}\|p-\qcr\| +\oscf
\end{align*}
for the CR-NCFEM.
The notation $A \lesssim B$ abbreviates the inequality
$A\leq CB$ with a mesh-size independent generic constant $C>0$.
The constant $C$ may depend on the minimal angle in the triangulation but not on 
the local mesh-size.
The best approximation result leads to the comparison
\begin{align*}
 \|\gradnc(u-\ucr)\|+\|p-\pcr\|
 \lesssim \|\nabla(u-\um)\|+\|p-\pmini\| + \|h_\tri f\|
\end{align*}
with the additional term $\|h_\tri f\|$ with 
the piecewise constant mesh-size $h_\tri$.

The $P_2 P_0$-FEM and the BR-FEM, from 
Figure~\ref{f:FEMsP2P0} and \ref{f:FEMsBR}, 
approximate the velocity by piecewise $P_2$ and some enriched $P_1$
functions and the pressure by piecewise constant functions.
The conformity of the $P_2 P_0$-FEM and the inclusion
$\Vbr\subseteq\Vp$ for the underlying finite element spaces
of the velocity approximation of BR-FEM and 
$P_2 P_0$-FEM imply
\begin{align*}
  \|\nabla(u-\up)\| + \|p-\pp\|
 \lesssim
 \|\nabla( u-\ubr)\| + \|p-\pbr\|.
\end{align*}
Since there exist examples where the convergence of the $P_2P_0$-FEM is of 
second order and the BR-FEM is a first order method the converse direction of 
this estimate cannot be expected to hold in general (see Remark 
\ref{r:BRlessP2P0})
The use of a conforming companion of the non-conforming solution $\ucr\in\Vcr$
of the CR-NCFEM yields
\begin{align*}
  \|\nabla(u-\ubr)\|+\|p-\pbr\|
   \lesssim \|\gradnc(u-\ucr)\|+\|p-\pcr\|.  
\end{align*}
Altogether, the main comparison results of this paper read
\begin{equation}\label{e:compCompleteIntro}
\begin{aligned}
 \|\nabla(u-\up)\|+\|p-\pp\|
 &\lesssim \|\nabla(u-\ubr)\|+\|p-\pbr\|\\
 &\lesssim \|\gradnc(u-\ucr)\|+\|p-\pcr\|\\
 &\lesssim \|\nabla(u-\um)\|+\|p-\pmini\|+\|h_\tri f\|.
\end{aligned}
\end{equation}

Furthermore this paper discusses the pressure approximation
by piecewise constant functions and by continuous piecewise
affine functions.
Theorem~\ref{t:pCRNlesspP1} proves that 
\begin{align*}
 \|p-p_h\|\lesssim\|\nabla(u-u_H)\|+\|p-p_H\|+\oscf
\end{align*}
does \emph{not} hold in general for solutions $(u_h,p_h)$
and $(u_H,p_H)$ of FEMs with piecewise constant
resp.\ continuous piecewise affine approximations of the pressure.
On the other hand,
the continuity of the pressure approximation 
is not a natural restriction and causes 
that
\begin{align*}
 \|p-p_H\|\lesssim \|\gradnc(u-u_h)\|+\|p-p_h\|
\end{align*}
does {\em not} hold in general.

Additionally the paper includes a comparison of CR-NCFEM
with a pseudostress approximation.

\medskip
All of the results are proven by {\it medius analysis}.
This means that arguments from a~posteriori techniques
lead to a~priori results. The notation medius analysis
was introduced in \cite{Gudi10}
and this technique leads to results which rely
on minimal regularity of the weak solution (i.e. $f\in L^2(\Omega)$) and hold
even for arbitrary coarse meshes.

\medskip
For all four considered FEMs a three-dimensional
extension \cite{BoffiBrezziFortin13} exists. In this situation all the 
arguments of this paper are applicable and the results remain true.

\medskip
The remaining parts of this paper are organised as follows. 
Section~\ref{s:preliminaries} introduces the
FEMs as well as underlying triangulations, corresponding operators, and other 
notation.
Section~\ref{s:CRbestapprox} performs a medius analysis of the CR-NCFEM. The
comparison results are stated and proven in Section~\ref{s:comparisonresults}. In
particular Subsection~\ref{ss:P1results} presents the comparison between CR-NCFEM
and MINI-FEM,  
Subsection~\ref{ss:P2results} is devoted to the comparison between
$P_2 P_0$-FEM, BR-FEM and CR-NCFEM. 
The comparison of the pressure approximations is performed in Subsection~\ref{ss:pressure}
and the inclusion of 
further methods is discussed in Subsection~\ref{ss:furtherMethods}.
Section~\ref{s:numerics} illustrates the behaviour
of the four FEMs from Figure~\ref{f:FEMs} and 
Figure~\ref{f:FEMP2s} in numerical experiments.
Subsection~\ref{ss:conclusions} summarises the paper 
with some conclusions.

\medskip
Throughout this paper, standard notation on Lebesgue
and Sobolev spac\-es is employed and 
$\|\bullet\|:=\|\bullet\|_{L^2(\Omega)}$ abbreviates 
the $L^2$ norm.
The formula $A\lesssim B$ abbreviates an inequality
$A\leq C B$ for some mesh-size independent, positive 
generic constant $C$; $A\approx B$ abbreviates
$A\lesssim B\lesssim A$.
The space $\C$ denotes the space of continuous functions and
$\Cnull:=\C\cap H^1_0(\Omega)$ the space of continuous
functions with homogeneous Dirichlet boundary conditions.
$A:B$ denotes the scalar product $A:B=\sum_{j,k=1}^2 A_{jk} B_{jk}$
for $A,B\in\mathbb{R}^{2\times 2}$.

\section{Preliminaries}\label{s:preliminaries}

This section introduces precise definitions of the 
Stokes equations and the FEMs under consideration.

\subsection{Stokes Equations}
Given a right-hand side $f\in L^2(\Omega;\mathbb{R}^2)$
in some polygonal Lipschitz domain,
the weak formulation of \eqref{e:StokesProblem} seeks $u\in
H^1_0(\Omega;\mathbb{R}^2)$
and $p\in L^2_0(\Omega)$ with 
\begin{equation}\label{e:weakForm}
\begin{aligned}
 \int_\Omega \nabla u :\nabla v\dx -\int_\Omega p\operatorname{div}v\dx
 &= \int_\Omega f\cdot v\dx
 \qquad&&\text{for all }v\in H^1_0(\Omega),\\
 \int_\Omega q\operatorname{div} u \dx &= 0
 \qquad&&\text{for all }q\in L^2_0(\Omega).
\end{aligned}
\end{equation}

\subsection{Triangulations}
A shape-regular triangulation $\tri$ of a bounded
Lipschitz domain $\Omega\subseteq \mathbb{R}^2$ is a set of triangles $T\in \tri$
such that $\overline{\Omega}=\cup\tri$ and any two distinct 
triangles are either disjoint or share exactly 
one common edge or one vertex.
Let $\N$ denote the set of vertices of $\tri$ 
and $\E$ the set of edges. 
The set of interior nodes is defined by $\N(\Omega):=\N\cap\Omega$
and the set of interior edges by $\E(\Omega):=\{E\in\E\mid E\not\subseteq \partial\Omega\}$.
Let $\N(T)$ denote the nodes of a triangle $T\in\tri$, $\tri(z):=\{T\in \tri\mid 
z\in \N(T)\}$ the elements which contain the node $z\in \N$, and $|\tri(z)|$ the 
number of elements in $\tri(z)$. 
Let 
\begin{align*}
\begin{array}{rl}
P_k(T;\R^m) &:= \{v_k:T\rightarrow \R^m\;|\;\forall j=1,\ldots,m,\text{ the
component}\\
      &\qquad\qquad\;v_k(j) \text{ of }v_k
\text{ is a polynomial of total degree}\leq k\},\\
P_k(\tri;\R^m) &:= \{v_k:\Omega\rightarrow \R^m\;| \;
\forall T\in\tri,v_k|_T \in P_k(T;\R^m)\}
\end{array}
\end{align*}
denote the set of piecewise polynomials
and abbreviate $P_k(\tri)=P_k(\tri;\R)$. The $L^2$ projection 
\[{\Pi_0:L^2(\Omega;\R^m)\rightarrow P_0(\tri;\R^m)}\] is given by  
$\tri$-piecewise constant functions or vectors 
$(\Pi_0 f)\vert_T=\fint_T f \dx:=\int_T f\dx/\vert T\vert$ 
for all $T\in\tri$ with area $\vert T\vert$ 
and all $f\in L^2(\Omega;\mathbb{R}^m)$. 
Let $h_\tri\in P_0(\tri)$ denote the piecewise constant mesh-size 
with $h_\tri\vert_T:=\operatorname{diam}(T)$ for all $T\in\tri$.

For piecewise affine functions $v_h\in P_1(\tri)$ the 
$\tri$-piecewise gradient $\gradnc v_h$ with \linebreak
$(\gradnc v_h)\vert_T=\nabla (v_h\vert_T)$ for all $T\in\tri$ 
and, accordingly, $\divnc\tau_h$ for $\tau_h\in P_1(\tri;\mathbb{R}^2)$
exists with $\gradnc v_h\in P_0(\tri;\mathbb{R}^{2})$
and $\divnc\tau_h\in P_0(\tri)$.

The oscillations of $f\in L^2(\Omega)$ read 
$\oscf:= \|h_\tri(f-\Pi_0 f)\|$.

\subsection{Finite Element Methods}\label{ss:defFEMs}
This section presents different finite element methods that have a piecewise
polynomial approximation of the velocity field. The pressure is approximated
with either piecewise constants or continuous piecewise affine functions. 
All methods are first-order accurate for a general smooth
solution $(u,p)\in H^2(\Omega;\mathbb{R}^2)\times H^1(\Omega)$. 
\paragraph{CR-NCFEM}
The $P_1$ non-conforming finite element method CR-NCFEM
after Crou\-zeix and Raviart \cite{CR} employs the space
\begin{align*}
 \CR:=\{\vcr\in  P_1&(\tri)\;|\;\vcr
\text{ is continuous }
\text{at midpoints of interior }\\
&\text{edges and vanishes at midpoints of }
\text{boundary edges}\}.
\end{align*}
The velocity is approximated in the space
\begin{align*}
 \Vcr:=\CR\times\CR.
\end{align*}
The CR-NCFEM seeks $(\ucr,\pcr)\in\Vcr\times \big(P_0(\tri)\cap L^2_0(\Omega)\big)$
such that
\begin{equation}\label{e:CRproplem}
\begin{aligned}
 \int_\Omega\gradnc\ucr:\gradnc\vcr\dx-\int_\Omega\pcr\divnc\vcr\dx
  &= \int_\Omega f\cdot\vcr\dx,\\
\int_\Omega\qcr\divnc\ucr\dx &= 0
\end{aligned}
\end{equation}
for all $\vcr\in\Vcr$ and $\qcr\in \big( P_0(\tri)\cap L^2_0(\Omega)\big)$; The
CR-NCFEM is inf-sup stable \cite{CR}.

\paragraph{MINI-FEM}
In the MINI-FEM \cite{ABF} the continuous
piecewise affine approximation for the velocity 
is enlarged with cubic bubble
functions, namely by elements of
\begin{align*}
 \B:=\{\psi\in P_3(\tri)\cap\Cnull\;\vert\;
   \forall &T=\operatorname{conv}\{a,b,c\}\in\tri\;\\
& \exists \alpha_T\in\mathbb{R}:\psi\vert_T=\alpha_T\varphi_a\varphi_b\varphi_c
\},
\end{align*}
where $\varphi_a$ (resp.\ $\varphi_b$, $\varphi_c$) 
is the piecewise affine nodal basis function of the node
$a$ (resp.\ $b$, $c$).
The MINI-FEM space for the velocity reads
\begin{align*}
\Vm:=\Big(\big(P_1(\tri)\cap \Cnull\big)
+\B\Big)^2.
\end{align*}
The MINI-FEM seeks $(\um,\pmini)\in\Vm\times (P_1(\tri)\cap\C\cap 
L^2_0(\Omega))$
with
\begin{equation}\label{e:mini}
\begin{aligned}
 \int_\Omega\nabla\um:\nabla\vm\dx
  -\int_\Omega\pmini\operatorname{div}\vm\dx &=\int_\Omega f\cdot\vm\dx,\\
\int_\Omega\qm\operatorname{div}\um\dx&=0
\end{aligned} 
\end{equation}
for all $\vm\in\Vm$ and $\qm\in (P_1(\tri)\cap\C\cap L^2_0(\Omega))$;
The MINI-FEM is inf-sup stable \cite{ABF}.

\paragraph{$P_2 P_0$-FEM}
The $P_2 P_0$-FEM seeks 
$\up\in \Vp:=\big(P_2(\tri)\cap
\Cnull\big)^2$ and 
$\pp\in P_0(\tri)\cap L^2_0(\Omega)$
with 
\begin{equation}\label{e:P2}
\begin{aligned}
 \int_\Omega\nabla\up:\nabla\vp\dx
  -\int_\Omega\pp\operatorname{div}\vp\dx &=\int_\Omega
f\cdot\vp\dx,\\
\int_\Omega\qp\operatorname{div}\up\dx&=0 
\end{aligned}
\end{equation}
for all $\vp\in\Vp$ and all $\qp\in P_0(\tri)\cap L^2_0(\Omega)$; The 
$P_2 P_0$-FEM is inf-sup stable \cite{BBDDFF06}.

\paragraph{BR-FEM}
The BR-FEM after Bernardi and Raugel \cite{BR} is a modification of the 
$P_2 P_0$-FEM. It is sometimes also called reduced $P_2 P_0$-FEM \cite{BBDDFF06}.
For a node $a\in \N$, let $\varphi_a$ denote the $P_1$ nodal basis function
and for an edge $E\in\E$, let 
$\nu_E$ denote the outer unit normal.
The space of edge bubbles reads 
\begin{align*}
 \B_\E:=\{\psi\in \big(P_2(\tri)\cap\Cnull\big)^2\;\vert\;
   \forall E=\operatorname{conv}\{a,b\}\in\E\;
  \exists &\alpha_E\in\mathbb{R}\\
 & \psi\vert_E=\alpha_E\varphi_a\varphi_b \nu_E\}.
\end{align*}
The BR-FEM approximation seeks $\ubr\in
\Vbr:= \big(P_1(\tri)\cap C_0(\Omega)\big)^2\oplus\B_\E$ and $p\in
P_0(\tri)\cap L^2_0(\Omega)$ with 

\begin{equation}\label{e:BR}
\begin{aligned}
 \int_\Omega\nabla\ubr:\nabla\vbr\dx
  -\int_\Omega\pbr\operatorname{div}\vbr\dx &=\int_\Omega
f\cdot\vbr\dx,\\
\int_\Omega\qbr\operatorname{div}\ubr\dx&=0
\end{aligned} 
\end{equation}
for all $\vbr\in \Vbr$ and all $\qbr\in P_0(\tri)\cap L^2_0(\Omega)$; The
BR-FEM is inf-sup stable \cite{BR}.

\subsection{Conforming Companions}
The design of three conforming companions to any $\vcr\in\Vcr$  begins with the
map $J_1: \CR\rightarrow P_1(\tri)\cap C_0(\Omega)$ defined by
\begin{align*}
 J_1\vcr:=\sum_{z\in \N(\Omega)}|\tri(z)|^{-1}\sum_{T\in
\tri(z)}\vcr|_T(z)\,\varphi_z,
\end{align*}
where $\varphi_z$ denotes the conforming nodal basis function.
For a given edge $E:=\operatorname*{conv}\{a,b\}\in \E$ let
$b_E:=6\varphi_a\varphi_b$ denote the edge bubble function. Then the operator
$J_2:\CR\rightarrow P_2(\tri)\cap C_0(\Omega)$ is given by
\[J_2\vcr:=J_1\vcr+\sum_{E\in\E(\Omega)}\left(\fint_E(\vcr-J_1\vcr)\ds\right)b_E.\]
For any triangle $T\in \tri$ with $T=\operatorname*{conv}\{a,b,c\}$ define the
element bubble function  $b_T:=60\varphi_a\varphi_b\varphi_c$. The operator
$J_3:\CR\rightarrow P_3(\tri)\cap C_0(\Omega)$ is given by
\[J_3\vcr:=J_2\vcr+\sum_{T\in\tri}\left(\fint_T(\vcr-J_2\vcr)\dx\right)b_T.\]
%
\begin{lemma}[{\cite{CGS13EV}}]\label{l:lemJ3}
The operators $J_k:\CR\to(P_k(\tri)\cap\Cnull)$, $k=1,2,3$, defined above satisfy 
the conservation properties
\begin{subequations}\label{e:conservationProperty}
\end{subequations}
\begin{align}
 \int_T\gradnc(\vcr - J_k\vcr)\dx 
&= 0
&&\quad\text{for all } T\in\tri\text{ and } k=2,3,\\
 \int_T(\vcr-J_3\vcr)\dx&=0\quad&&\quad\text{for all }T\in\tri 
\end{align}
and the  approximation and stability properties for $k=1,2,3$
\begin{equation}\label{e:approxstabilJ3}
\begin{aligned}
\|h_\tri^{-1}(\vcr-J_k\vcr)\|&\approx\|\gradnc(\vcr-J_k\vcr)\|\\
 &\approx \min_{\varphi\in H^1_0(\Omega)} \|\gradnc(\vcr-\varphi)\|\\
&\leq\|\gradnc\vcr\|.
\end{aligned}
\end{equation}
\end{lemma}
\section{Medius Analysis for CR-NCFEM}\label{s:CRbestapprox}
This section states and proves a best-approximation
result for CR-NCFEM.
\begin{theorem}[best-approximation result]\label{t:CRbestapprox}
Any $\vcr\in\Vcr$ and $\qcr\in P_0(\tri)\cap L^2_0(\Omega)$ satisfy
\begin{align*}
 \|\gradnc(u-\ucr)\|+\|p-\pcr\|
 \lesssim \|\gradnc(u-\vcr)\|+\|p-\qcr\| +\oscf.
\end{align*}
\end{theorem}
The error analysis of \cite{CR} employs a Strang-Fix decomposition.
To obtain an error estimate this approach requires 
$u\in H^2(\Omega)$ and $p\in H^1(\Omega)$. For the 
medius analysis of Theorem~\ref{t:CRbestapprox} this 
assumption is dropped.

\begin{proof}[Proof of Theorem \ref{t:CRbestapprox}]
The non-conforming interpolation operator denoted by \linebreak
$\Inc:H^1_0(\Omega;\mathbb{R}^2)\to \Vcr$
is defined by 
\begin{align*}
 \Inc v (\operatorname{mid}(E)) := \fint_E v \ds 
 \qquad\text{for all }v\in H^1_0(\Omega;\mathbb{R}^2)
 \text{ and all }E\in\E(\Omega).
\end{align*}
The error of the velocity satisfies
\begin{align*}
 \|\gradnc(u-\ucr)\|^2
 &=\int_\Omega\gradnc(u-\Inc u):\gradnc(u-\ucr)\dx\\
  &\qquad +\int_\Omega\gradnc(\Inc u-\ucr):\gradnc(u-\ucr)\dx.
\end{align*}
In order to estimate the second term consider the function $J_3\wcr$
for $\wcr:=\Inc u-\ucr$ from
Lemma~\ref{l:lemJ3}.
Since $\divnc\wcr=0$,
the second term reads
\begin{align*}
 \int_\Omega&\gradnc(\Inc u-\ucr):\gradnc(u-\ucr)\dx\\
  &=\int_\Omega\nabla u:\gradnc(\wcr-J_3\wcr)\dx
   +\int_\Omega f\cdot (J_3\wcr-\wcr)\dx\\
  &\quad +\int_\Omega p\operatorname{div}J_3\wcr\dx.
\end{align*}
Since $\Pi_0 \nabla(J_3\wcr) = \gradnc\wcr$, this equals
\begin{align*}
\int_\Omega (\nabla u-\Pi_0 &\nabla u):\gradnc(\wcr-J_3\wcr)\dx
+\int_\Omega(f-\Pi_0 f)(J_3\wcr-\wcr)\dx\\
&\quad+\int_\Omega(p-\Pi_0 p)\operatorname{div} J_3\wcr \dx\\
&\leq\|\gradnc(u-\Inc u)\|\;\|\gradnc(\wcr-J_3\wcr)\|\\
 &\quad +\|\gradnc(J_3\wcr-\wcr)\|\oscf
  +\|p-\Pi_0 p\|\;\|\nabla J_3\wcr\|.
\end{align*}
The stability of $J_3$ leads to 
\begin{align*}
\int_\Omega&\gradnc(u-\ucr):\gradnc(\Inc u-\ucr)\dx\\
&\lesssim\big(\|\gradnc(u-\Inc u)\|+\oscf+\|p-\Pi_0 p\|\big)\;\|\gradnc\wcr\|.
\end{align*}
This implies
\begin{align*}
 \|\gradnc(u-\ucr)\|\lesssim\|\gradnc(u-\Inc u)\|+\|p-\Pi_0 p\|+\oscf.
\end{align*}
For the error of the pressure the discrete inf-sup condition implies that there
exists $\vcr\in\Vcr$ with $\|\gradnc\vcr\|=1$ such that
\begin{align*}
 \|\pcr-\Pi_0 p\|
&\lesssim\int_\Omega(\pcr-\Pi_0 p)\operatorname{div}\vcr\dx.
\end{align*}
The integral mean property $\Pi_0\nabla J_3\vcr = \gradnc\vcr$ implies
\begin{align*}
 \|\pcr-\Pi_0 p\|
&= \int_\Omega\gradnc\ucr:\gradnc\vcr\dx
 -\int_\Omega f\cdot\vcr\dx \\
&\quad-\int_\Omega\Pi_0 p \divnc J_3\vcr\dx\\
&=\int_\Omega\gradnc(\ucr- u):\gradnc\vcr\dx
+\int_\Omega f\cdot(J_3\vcr-\vcr)\dx\\
&\quad+\int_\Omega(p-\Pi_0 p)\operatorname{div}J_3\vcr\dx\\
&\quad+\int_\Omega(\nabla u-\Pi_0\nabla u):\gradnc(\vcr-J_3\vcr)\dx.
\end{align*}
The approximation and stability properties of $J_3$
and $\Pi_0\nabla J_3\vcr = \gradnc\vcr$ imply
\begin{align*}
 \|\pcr-\Pi_0 p\|
&\lesssim\|\gradnc(u-\ucr)\|+\oscf+\|p-\Pi_0 p\|.
\end{align*}
This concludes the proof.
\end{proof}

%

\section{Comparison Results}\label{s:comparisonresults}
This section establishes comparisons between the FEMs introduced in
Subsection~\ref{ss:defFEMs}.

\subsection{CR-NCFEM versus MINI-FEM}
\label{ss:P1results}
This section compares CR-NCFEM 
with MINI-FEM.
\begin{theorem}\label{t:CRLessMini}
The solution  
$(\ucr,\pcr)\in \Vcr\times \big(P_0(\tri)\cap L^2_0(\Omega)\big)$ of 
the CR-NCFEM and the
solution $(\um,\pmini)\in \Vm\times 
\big(P_1(\tri)\cap C_0(\Omega)\cap L^2_0(\Omega)\big)$ of the MINI-FEM satisfy
\begin{align*}
 \|\gradnc(u-\ucr)\|+\|p-\pcr\|
 \lesssim \|\nabla(u-\um)\|+\|p-\pmini\| + \|h_\tri f\|.
\end{align*}
\end{theorem}

\begin{remark}
Since CR-NCFEM has a piecewise constant and the MINI-FEM has
a globally continuous and piecewise affine pressure approximation, the converse 
estimate cannot be expected to hold in general, cf.\ Theorem~\ref{t:pP1NlessCR}. 
The question remains open whether the unnatural continuous 
or the natural discontinuous 
pressure approximation is better. 
\end{remark}

\noindent The following lemma is essential in the proof of 
Theorem~\ref{t:CRLessMini}.

\begin{lemma}\label{l:LinEquivAll}
Let $\um=\ulin+\ub\in\Vm$ denote the solution of \eqref{e:mini} which is
split into $\ulin\in (P_1(\tri)\cap \Cnull)^2$ and $\ub\in\B^2$.
Then it holds
\begin{align*}
 \|\nabla(u-\ulin)\|\lesssim\|\nabla(u-\um)\|+\|p-\pmini\|+\oscf.
\end{align*}
\end{lemma}

\begin{proof}The arguments of \cite{Russo96} 
determine the bubble part $\ub$ with a general function $f\in L^2(\Omega)$. 
For $b_T=\varphi_a\varphi_b\varphi_c\in\B$ with
the piecewise affine nodal basis functions $\varphi_a,\varphi_b,\varphi_c$
of $a,b,c$
and $T=\operatorname{conv}\{a,b,c\}\in\tri$
this
yields
\begin{align*}
\ub\vert_T = \int_T b_T\dx\, b_T\, 
		  (\Pi_0 f - \nabla \pmini)/\|\nabla
b_T\|^2+&\int_T(f-\Pi_0f)b_T\dx\, b_T/\|\nabla b_T\|^2.
\end{align*}
This implies
\begin{align*}
 \|\nabla \ub\|\lesssim \|h_\tri 
               (\Pi_0 f - \nabla\pmini)\|+\oscf.
\end{align*}
It holds
\begin{align*}
 \Delta\ub\vert_T =& \int_T b_T\dx\, (\nabla \pmini-\Pi_0 f)/\|\nabla b_T\|^2
\;\Delta b_T\\
&+\int_T(f-\Pi_0 f)b_T\dx\Delta b_T/\|\nabla b_T\|^2. 
\end{align*}
Since $\nabla \pmini-\Pi_0 f$ is piecewise constant and $\int_T b_Tdx\Delta
b_T\approx 1$ the previous two displayed formulas result in 
\begin{align*}
 \|\nabla \ub\|
  \lesssim \|h_\tri (\nabla\pmini-\Pi_0 f-\Delta \ub)\|+\oscf.
\end{align*}
The bubble-technique of \cite{Ve1996} leads
to the efficiency
\begin{align*}
 \|h_\tri (\nabla\pmini-\Pi_0 f-\Delta \ub)\|
  \lesssim \|\nabla(u-\um)\|+\|p-\pmini\|+\oscf.
\end{align*}
This and a triangle inequality conclude the proof.
\end{proof}

\begin{proof}[Proof of Theorem \ref{t:CRLessMini}]
Theorem~\ref{t:CRbestapprox} implies 
for $\um=\ulin+\ub$ with $\ulin\in \big(P_1(\tri)\cap C_0(\Omega)\big)^2$
and $\ub\in\B^2$ that
\begin{align*}
  \|\gradnc(u&-\ucr)\|+\|p-\pcr\|\\
  &\lesssim \|\nabla (u-\ulin)\| + \|p-\Pi_0 \pmini\| + \oscf\\
  &\lesssim \|\nabla (u-\ulin)\| + \|p-\pmini\| + \|\pmini - \Pi_0 \pmini\|
+ \oscf.
\end{align*}
Since $\Delta_{\operatorname{NC}} \ulin=0$, a Poincar\'e inequality yields
\begin{align*}
 \|\pmini - \Pi_0 \pmini\|
 \leq \|h_\tri \nabla\pmini\|
 \leq \|h_\tri(\nabla\pmini+f+\Delta_{\operatorname{NC}}\ulin)\| + \|h_\tri f\|.
\end{align*}
The efficiency \cite{Verfuerth89} of
$\|h_\tri(\nabla\pmini+f+\Delta_{\operatorname{NC}}\ulin)\|$ reads
\begin{align*}
 \|h_\tri(\nabla\pmini + & f + \Delta_{\operatorname{NC}}\ulin)\|\\
 &\lesssim \|\nabla(u-\ulin)\|+\|p-\pmini\| + \oscf.
\end{align*}
This and Lemma~\ref{l:LinEquivAll} conclude the proof.
\end{proof}


\subsection{Comparison of $P_2 P_0$-FEM, BR-FEM and CR-NCFEM}\label{ss:P2results}

First, Theorem~\ref{t:P2lessBR} and Theorem~\ref{t:BRlessCR} 
of this section complete 
the comparisons \eqref{e:compCompleteIntro}.
Afterwards, Theorem~\ref{t:CRlessP2} discusses 
converse directions of those comparisons.

\begin{theorem}\label{t:P2lessBR}
The solution $(\ubr,\pbr)\in \Vbr\times \big(P_0(\tri)\cap L^2_0(\Omega)\big)$
of the BR-FEM and the solution $(\up,\pp)\in\Vp\times \big(P_0(\tri)\cap
L^2_0(\Omega)\big)$ of the $P_2P_0$-FEM
satisfy
\begin{align*}
 \|\nabla(u-\up)\| + \|p-\pp\|
 \lesssim
 \|\nabla( u-\ubr)\| + \|p-\pbr\|.
\end{align*}
\end{theorem}

\begin{proof}
This follows from the conformity and stability of the $P_2 P_0$-FEM 
and $\Vbr\subseteq \Vp$.
\end{proof}

\begin{remark}\label{r:BRlessP2P0}
The $P_2 P_0$-FEM and the BR-FEM approximate the velocity field
with different polynomial order. In the case of vanishing pressure 
$p=0$ and smooth regularity, the $P_2 P_0$-FEM converges like a second-order method, whereas
the BR-FEM remains a first-order method. 
Thus, the converse estimate cannot be expected to hold in general.
\end{remark}

\begin{theorem}\label{t:BRlessCR}
 It holds 
\begin{align*}
 \|\nabla(u-\ubr)\|+\|p-\pbr\|\lesssim \|\gradnc(u-\ucr)\|+\|p-\pcr\|.
\end{align*}
\end{theorem}

\begin{proof}
Consider the operator $J_1:\CR\rightarrow P_1(\tri)\cap C_0(\Omega)$ from 
Lemma~\ref{l:lemJ3}. Since the BR-FEM is a conforming FEM, it holds
\begin{align*}
 \|\nabla(u-\ubr)\|+\|p-\pbr\|\lesssim
\|\nabla(u-J_1\ucr)\|+\|p-\pcr\|. 
\end{align*}
(Here, the operator $J_1$ is applied componentwise.)
The operator $J_1$ satisfies 
\[\|\nabla(u-J_1\ucr)\|\leq
\|\gradnc(u-\ucr)\|+\|\gradnc(\ucr-J_1\ucr)\|\lesssim
\|\gradnc(u-\ucr)\|.\]
This concludes the proof.
\end{proof}


\begin{theorem}\label{t:CRlessP2}
It holds
\begin{align}\label{e:CRlessP2}
 \|\gradnc( u-\ucr)\| + \|p-\pcr\| \lesssim& \|\nabla(u-\up)\| +\|p-\pp\|\\
\notag &+\oscf+\|\nabla\up-\Pi_0\nabla \up\|
\end{align}
as well as
\begin{align}\label{e:CRlessPBR}
 \|\gradnc( u-\ucr)\| + \|p-\pcr\| \lesssim& \|\nabla(u-\ubr)\| +\|p-\pbr\|\\
\notag &+\oscf+\|\nabla\ubr-\Pi_0\nabla \ubr\|.
\end{align}
\end{theorem}

\begin{proof}
Theorem \ref{t:CRbestapprox} immediately leads to
\begin{align*}
  \|\nabla(u-\ucr)\|+\|p-\pcr\|\lesssim
\|\gradnc(u-\Inc \up)\|+\|p-\pp\|+\oscf.
\end{align*}
A triangle inequality and $\gradnc\Inc \up=\Pi_0\nabla \up$ yield
\[\|\gradnc(u-\Inc \up)\|\leq
\|\nabla(u-\up)\|+\|\nabla \up-\Pi_0\nabla \up\|.\] 
This completes the proof of \eqref{e:CRlessP2}.\\
The same arguments prove the second statement.
\end{proof}

\subsection{Non-comparability of continuous and discontinuous pressure}
\label{ss:pressure}
This section compares FEMs with  pressure approximations
in $P_1(\tri)\cap\C\cap L^2_0(\Omega)$ with
FEMs with pressure approximations in $P_0(\tri)\cap L^2_0(\Omega)$.
The subsequent theorems state that 
FEMs with discontinuous pressure approximations are {\em not} 
comparable with FEMs with continuous pressure approximation.

\begin{theorem}\label{t:pP1NlessCR}
 Let $(u_h,p_h)$ denote the discrete solution of the Stokes equations for any
finite element method which approximates the pressure $p$ 
with continuous piecewise affine
functions $p_h\in P_1(\tri)\cap \C\cap L^2_0(\Omega)$. 
Let $(u_H,p_H)$ denote the solution of the CR-NCFEM, the $P_2 P_0$-FEM 
or the BR-FEM.
Then, in general,
\begin{align*}
 \|p-p_h\|\not\lesssim \|\gradnc(u-u_H)\|+\|p-p_H\|.
\end{align*}
\end{theorem}

\begin{proof}
On the rhombus $\Omega=\mathrm{conv}\{(1,0),(0,1),(-1,0),(0,-1)\}$ 
define the right-hand side $f_\varepsilon\in 
L^2(\Omega;\mathbb{R}^2)$
by $f_\varepsilon(x,y)=\varepsilon^{-1} (1,0)$ 
for $-\varepsilon\leq x\leq\varepsilon$ and $f_\varepsilon(x,y)=0$ otherwise.
Then $(u,p_\varepsilon)\in H^1_0(\Omega;\mathbb{R}^2)\times L^2_0(\Omega)$
with
$u\equiv0$ and
\begin{align*}
 p_\varepsilon(x,y):=\begin{cases}
             -1 & \text{ for }-1\leq x\leq-\varepsilon,\\
             x/\varepsilon & \text{ for }-\varepsilon\leq x\leq \varepsilon,\\
             1 & \text{ for }\varepsilon\leq x\leq1\\       
             \end{cases}
\end{align*}
satisfies
\begin{align*}
 \int_\Omega\nabla u:\nabla v\dx
    -\int_\Omega p_\varepsilon\operatorname{div}v\dx
  &=\int_\Omega f_\varepsilon\cdot v\dx,\\
 \int_\Omega q\operatorname{div} u\dx&=0.
\end{align*}
Let $\tri:=\{T_1,T_2\}$ be the triangulation with 
$T_1:=\operatorname{conv}\{(0,1),(0,-1),(1,0)\}$ and 
$T_2:=\operatorname{conv}\{(0,-1),(0,1),(-1,0)\}$.
The  solutions of the CR-NCFEM for the right-hand side $f_\varepsilon\in
L^2(\Omega;\mathbb{R}^2)$ then
reads $(\ucr,\pcr)\in\Vcr\times P_0(\tri)$ with
$\ucr\equiv0$ and $\pcr(x,y)=-1+\varepsilon/2+2\varepsilon^2/3$ 
for $-1\leq x\leq 0$
and $\pcr(x,y)=1-\varepsilon/2-2\varepsilon^2/3$ for $0\leq x\leq 1$.
This shows $\|u-\ucr\|=0$ and 
$\|p_\varepsilon-\pcr\|\to0$ for $\varepsilon\to0$.
On the other hand, 
symmetry arguments imply $p_h\vert_{\{0\}\times (-1,1)}=0$
and, hence, $\|p_\varepsilon-p_h\|\gtrsim1$.
This proves the assertion in the case that $(u_H,p_H)$
is the solution of the CR-NCFEM.
Since the $P_2 P_0$-FEM and the  BR-FEM are conforming,
the best-approximation property implies
for the solution $(u_H,p_H)$ 
of the $P_2 P_0$-FEM or the BR-FEM that
\begin{align*}
 \|\nabla(u-u_H)\|+\|p-p_H\|
 \lesssim \|p-\pcr\|.
\end{align*}
This concludes the proof.
\end{proof}

\begin{theorem}\label{t:pCRNlesspP1}
 Let $(u_h,p_h)$ denote the discrete solution of the Stokes equations for any
(conforming) FEM which approximates the pressure $p$ with piecewise affine
functions
and let $(u_H,p_H)$ be the solution of the CR-NCFEM, $P_2 P_0$-FEM or BR-FEM.
Then it holds
\begin{align}\label{e:CRnotlessMINI}
 \|p-p_H\|\not\lesssim\|\nabla(u-u_h)\|+\|p-p_h\|+\oscf.
\end{align}
\end{theorem}

\begin{proof}
On the rhombus $\Omega=\mathrm{conv}\{(1,0),(0,1),(-1,0),(0,-1)\}$ with
the triangulation $\tri:=\{T_1,T_2\}$ with 
$T_1:=\operatorname{conv}\{(0,1),(0,-1),(1,0)\}$ and 
$T_2:=\operatorname{conv}\{(0,-1),(0,1),(-1,0)\}$
and
the right-hand side $f\equiv(1,0)$, the exact solution
equals $u\equiv0$ and $p(x,y):=x$. This is approximated exactly by any
(conforming) FEM
with pressure approximation in $P_1(\tri)\cap\C\cap L^2_0(\Omega)$.
Hence, the right-hand
side
in \eqref{e:CRnotlessMINI} vanishes.
The fact that the exact pressure 
is not piecewise constant $p\not\in P_0(\tri)$, implies
for the left-hand side $\|p-p_H\|\neq0$.
\end{proof}
\begin{remark}\label{r:2/3convergence}
 Theorem 3.9 of \cite{ETX} states that if $(u,p)\in H^3(\Omega)^2\times 
H^2(\Omega)$ the pressure of the MINI-FEM even converges with the rate of 
$-3/4$. This can be seen in Subsection \ref{ss:colliding} and underlines the 
above result.  
\end{remark}

\subsection{Further finite elements}\label{ss:furtherMethods}

This section discusses how the Taylor-Hood-FEM,
stabilised $P_1 P_0$-FEM,
discontinuous Galerkin FEM and a pseudostress FEM 
can be included in the comparions \eqref{e:compCompleteIntro}.

\paragraph{Taylor-Hood}
The most common second-order FEM is the Taylor-Hood FEM 
\cite{BoffiBrezziFortin13} with $P_2$ velocity approximation and continuous
$P_1$ pressure approximation. 
The conformity of this method and Lemma~\ref{l:LinEquivAll} immediately shows
\begin{align*}
  \|\nabla(u-u_{\operatorname{TH}})\|
    + \|p-p_{\operatorname{TH}}\|
  \lesssim
    \|\nabla(u-u_{\operatorname{MINI}})\|
    + \|p-p_{\operatorname{MINI}}\|
\end{align*}
for the solution $(u_{\operatorname{TH}}, p_{\operatorname{TH}})$
from the Taylor-Hood FEM.
The comparison to the $P_2 P_0$-FEM, BR-FEM, and 
CR-NCFEM is not clear 
because of the different ansatz spaces for the pressure.

\paragraph{Stabilised $P_1 P_0$-FEM}
A direct consequence of Theorem~\ref{t:CRbestapprox}
is a comparison result for CR-NCFEM with the
stabilised $P_1 P_0$-FEM \cite{DW89}. 
Let $(u_h,p_h)\in P_1(\tri;\mathbb{R}^2)\times \big(P_0(\tri)\cap
L^2_0(\Omega)\big)$
denote the discrete solution 
of the stabilised $P_1 P_0$-FEM,
then the
CR-NCFEM is
superior in the sense that
\begin{align*}
 \|\gradnc(u-\ucr)\|+\|p-\pcr\|
 \lesssim \|\nabla(u-u_h)\| + \|p-p_h\| + \oscf.
\end{align*}

\paragraph{Discontinuous Galerkin FEM}
 For the weakly over penalised discontinuous Galerkin FEM (WOPSIP) 
 in \cite{BB13,BCGG}, a similar best-approximation result to 
 Theorem~\ref{t:CRbestapprox} is proven in \cite{BCGG}. 
 Since the norm $\|\bullet\|_h$ defined therein equals the 
 norm $\|\gradnc\bullet\|$ for the CR-NCFEM, the two 
 best-approximation results immediately yield equivalence of 
 CR-NCFEM and  WOPSIP discontinuous Galerkin FEM. 

\paragraph{Pseudostress FEM}
The pseudostress-velocity approximation of the stationary Stokes equations
\cite{CKP11}
seeks $\sigma_\PS\in\PS(\tri)
:= \{(\tau_1,\tau_2)\in(\operatorname{RT}_0(\tri)\times\operatorname{RT}_0(\tri))\;\vert\;
\tau_j\in H(\operatorname{div},\Omega)\text{ for}\linebreak j=1,2\text{ and }
\int_\Omega\operatorname{tr}(\tau_1,\tau_2)\dx=0\}$
%
and $u_\PS\in P_0(\tri;\R^2)$ with
\begin{align*}
 \int_\Omega \tau_\PS:\dev\sigma_\PS\dx 
   + \int_\Omega u_\PS\operatorname{div}\tau_\PS \dx
   &= 0 
   &&\text{for all }\tau_\PS\in\PS(\tri),\\
 \int_\Omega v_\PS\operatorname{div}\sigma_\PS\dx
   &= \int_\Omega f\cdot v_\PS\dx
   &&\text{for all }v_\PS\in P_0(\tri;\R^2).
\end{align*}

\begin{theorem}\label{t:compPSCR}
The pseudostress approximation satisfies
\begin{equation}\label{e:compPSCR}
\begin{aligned}
 \|\nabla u-\dev\sigma_\PS\| &+ \|p+\operatorname{tr}\sigma_\PS/2\|\\
 &\lesssim \|\gradnc(u-\ucr)\| + \|p-\pcr\| + \oscf\\
 &\lesssim \|\nabla u-\dev\sigma_\PS\| 
     + \|p+\operatorname{tr}\sigma_\PS/2\| + \|h_\tri f\|.
\end{aligned}
\end{equation}
\end{theorem}

\begin{proof}

Let $(\widetilde{u}_{\operatorname{CR}},\widetilde{p}_{\operatorname{CR}})
\in \Vcr\times P_0(\tri)\cap L^2_0(\Omega)$ denote the solution to the CR-NCFEM 
for the right-hand side $\Pi_0 f$.
Let $(\bullet -\operatorname{mid}(\tri))$ abbreviate
the function $(x-\operatorname{mid}(T))$ for 
$x\in T\in\tri$ with midpoint $\operatorname{mid}(T)$.
The solution $\sigma_\PS$ of the pseudostress approximation
of the Stokes equations \cite{CGS13} reads 
\begin{align}\label{e:PSrepresentation}
 \sigma_\PS = \gradnc \widetilde{u}_{\operatorname{CR}} 
    - \frac{\Pi_0 f}{2}\otimes (\bullet-\operatorname{mid}(\tri)) 
    - \widetilde{p}_{\operatorname{CR}} I_{2\times 2} . 
\end{align}
The deviatoric part of a matrix $A\in \R^{2\times2}$ reads 
$\dev(A):=A-\tr(A)/2I_{2\times 2}$. Since $\divnc 
\widetilde{u}_{\operatorname{CR}}\equiv 0$, it holds
\[\dev\sigma_\PS = \gradnc
\widetilde{u}_{\operatorname{CR}} 
- \dev\bigg(\frac{\Pi_0 f}{2}\otimes (\bullet-\operatorname{mid}(\tri))\bigg).\]
This implies
\begin{align*}
 \|\nabla u - \dev\sigma_\PS\|
 \leq \|\gradnc(u-\widetilde{u}_{\operatorname{CR}})\|
    + \|\dev(\Pi_0 f\otimes (\bullet-\operatorname{mid}(\tri)))\|/2.
\end{align*}
For the pressure approximation, the representation formula
\eqref{e:PSrepresentation}
leads to
\begin{align*}
 \|p+\operatorname{tr}\sigma_\PS/2\|
 \leq \|p-\widetilde{p}_{\operatorname{CR}}\|
     + \|\operatorname{tr}(\Pi_0 f\otimes
(\bullet-\operatorname{mid}(\tri)))\|/4.
\end{align*}
The orthogonal split in the trace and the deviatoric part
and the obvious estimate
$\lvert\bullet-\operatorname{mid}(\tri)\rvert\leq h_\tri$ in $\Omega$
lead to
\begin{align*}
 \|\dev(\Pi_0 f\otimes (\bullet-\operatorname{mid}(\tri)))\|
 + \|\operatorname{tr}(\Pi_0 f\otimes (\bullet-\operatorname{mid}(\tri)))\|
 \lesssim \|h_\tri f\|.
\end{align*}
The efficiency of $\|h_\tri f\|$ \cite{DDP95} leads to 
\begin{align*}
 \|h_\tri f\|
 \lesssim \|\gradnc(u-\ucr)\| + \| p-\pcr\| + \oscf.
\end{align*}
The discrete inf-sup condition for CR-NCFEM guarantees the existence of 
$\vcr\in\CR$ with $\|\gradnc\vcr\|=1$ and
\begin{align*}
 \|\pcr-\widetilde{p}_{\operatorname{CR}}\|
 &\lesssim \int_\Omega (\pcr-\widetilde{p}_{\operatorname{CR}})\divnc\vcr\dx\\
 &= \int_\Omega \gradnc(\ucr-\widetilde{u}_{\operatorname{CR}}):\gradnc\vcr\dx
      + \int_\Omega (f-\Pi_0f)\vcr\dx.
\end{align*}
This yields
\begin{align*}
 \|\pcr-\widetilde{p}_{\operatorname{CR}}\|
 \lesssim \|\gradnc(\ucr-\widetilde{u}_{\operatorname{CR}})\|
     + \oscf.
\end{align*}
Since
$\divnc\ucr=\divnc\widetilde{u}_{\operatorname{CR}}=0$,
the problem \eqref{e:CRproplem} implies
$\|\gradnc(\ucr-\widetilde{u}_{\operatorname{CR}})\| \lesssim \oscf$.
The combination of the previous inequalities gives the 
first inequality in \eqref{e:compPSCR}.
The same arguments yield the second inequality in \eqref{e:compPSCR}.\end{proof}

\section{Numerical illustration}\label{s:numerics}

This section illustrates the behaviour 
of the CR-NCFEM, the  MINI-FEM, the $P_2 P_0$-FEM and the BR-FEM
in two examples (Subsections \ref{ss:colliding}--\ref{ss:Lshape}).
Subsection~\ref{ss:conclusions} draws some conclusions 
from the numerical experiments.

\subsection{Colliding flow}\label{ss:colliding}
On the square domain $\Omega = (-1,1)\times(-1,1)$ 
with right-hand side $f\equiv 0$, the 
exact velocity to the corresponding boundary conditions is given by
$u(x,y)=(20xy^4-4x^5,20x^4y-4y^5)$ with pressure
$p(x,y)=120x^2y^2-20x^4-20y^4-32/6$.
The convergence history plot of Figure~\ref{f:colliding} shows
the errors
\begin{align*}
  \sqrt{\|\gradnc(u-u_h)\|^2 + \|p-p_h\|^2}
\end{align*}
for the discrete solutions $(u_h,p_h)$ of
the CR-NCFEM, the MINI-FEM, the  $P_2 P_0$-FEM and the BR-FEM
plotted against the degrees of freedom.
The four FEMs 
yield the same rate of convergence of $-1/2$ with respect
to the number of degrees of freedom and the errors 
are all of the same size.
\begin{figure}
 \centering
 \includegraphics[width=0.75\textwidth]{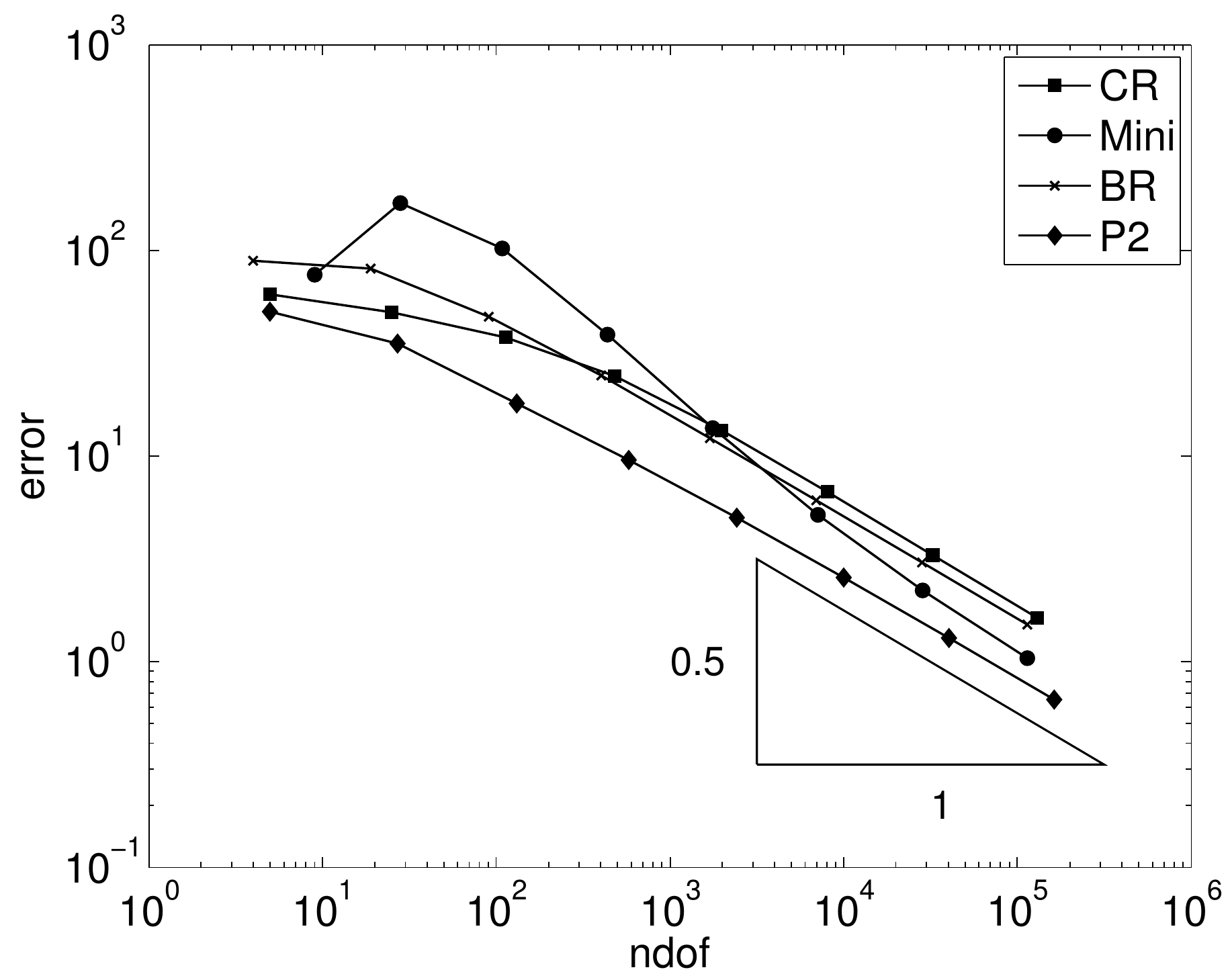}
 \caption{Convergence history for the colliding flow of 
    Subsection~\ref{ss:colliding}.}
 \label{f:colliding}
\end{figure}
\begin{figure}
 \centering
 \includegraphics[width=0.75\textwidth]{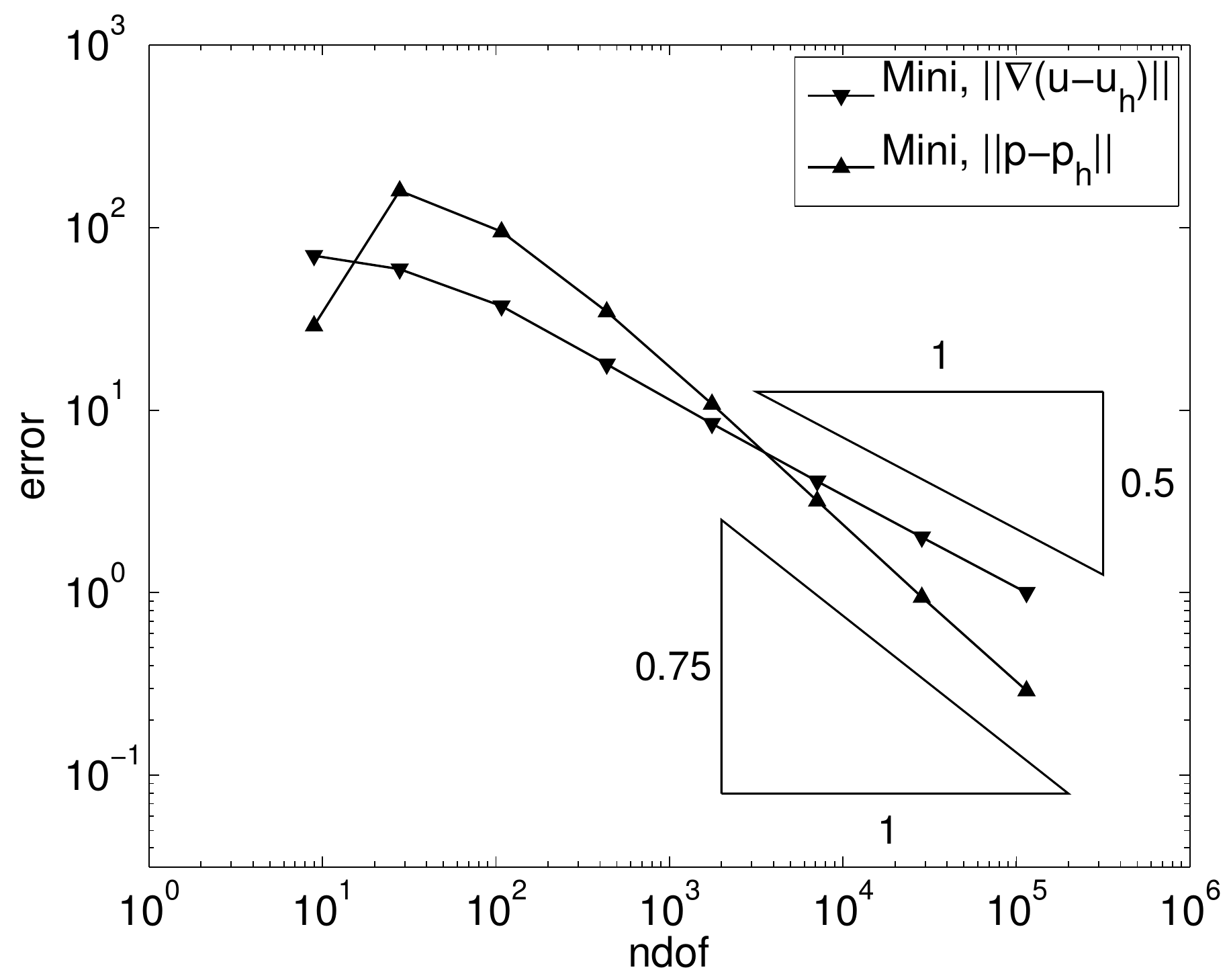}
 \caption{Convergence history for the pressure and velocity approximation of
    Mini-FEM for the colliding flow of
    Subsection~\ref{ss:colliding}.}
 \label{f:collidingPressure}
\end{figure}

Figure~\ref{f:colliding} shows that the MINI-FEM converges with an improved 
convergence rate in a pre-asymptotic range. This is due to the pressure 
approximation which converges with 
a rate of $-3/4$. This numerically highlights the result of Theorem 
3.9 
from \cite{ETX} which was stated in Remark \ref{r:2/3convergence}. Figure 
\ref{f:collidingPressure} clearly shows the difference of convergence rates for 
the pressure and velocity approximations. The pressure approximation converges 
with a rate of $-3/4$ whereas the velocity converges with a rate of $-1/2$ 
which also explains the overall convergence rate of $-1/2$ in the asymptotic 
regime.

\subsection{L-shaped domain}\label{ss:Lshape}
On the L-shaped domain 
$\Omega=(-1,1)\times(-1,1) \setminus \left( [0,1]\times[-1,0]\right)$
with right-hand side $f\equiv 0$,
the exact solution, with the characteristic singularity at the origin, for the 
corresponding boundary conditions, reads
\begin{align*}
 u(r,\vartheta) &= \begin{pmatrix}
                  r^\alpha ((1+\alpha)\sin(\vartheta)w(\vartheta)+\cos(\vartheta)w_\vartheta(\vartheta))\\
		  r^\alpha (-(1+\alpha)\cos(\vartheta)w(\vartheta)+\sin(\vartheta)w_\vartheta(\vartheta))
                 \end{pmatrix},\\
  p(r,\vartheta) &= -r^{\alpha-1} ((1+\alpha)^2 w_\vartheta (\vartheta) 
		  w_{\vartheta\vartheta\vartheta}(\vartheta)) /(1-\alpha)
\end{align*}
in polar coordinates with $\alpha = 0.54448373$ and
\begin{align*}
 w(\vartheta) = &(\sin((1+\alpha)\vartheta)\cos(\alpha\omega))/(1+\alpha)-\cos((1+\alpha)\vartheta)\\
    &-(\sin((1-\alpha)\vartheta)\cos(\alpha\omega))/(1-\alpha) + \cos((1-\alpha)\vartheta).
\end{align*} 
Figure~\ref{f:Lshape} shows the equivalence of all four FEMs
also for this example with a reduced convergence rate of $-1/4$ 
with respect to the number of degrees of freedom.
\begin{figure}
 \centering
 \includegraphics[width=0.75\textwidth]{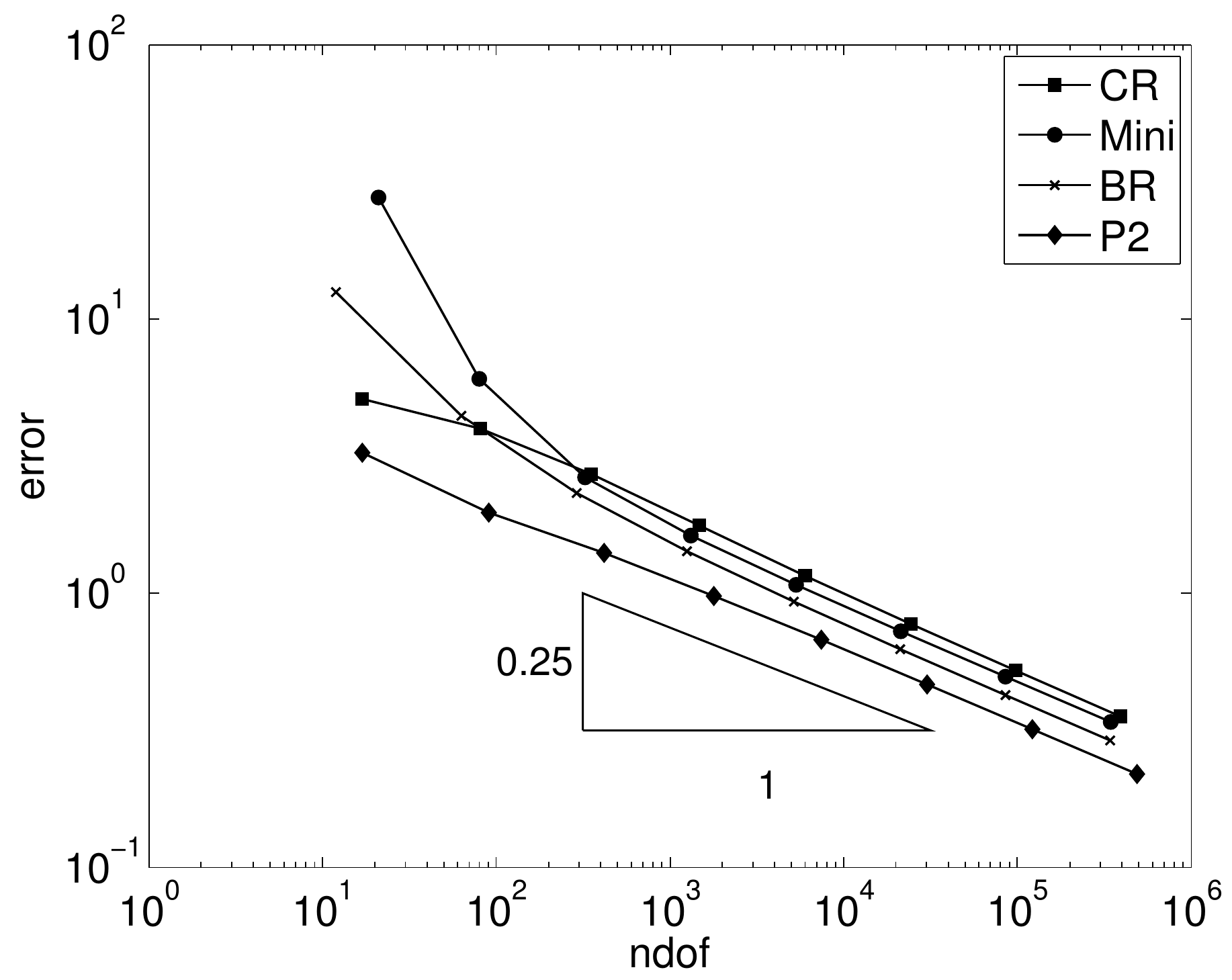}
 \caption{Convergence history for the L-shaped domain of 
    Subsection~\ref{ss:Lshape}.}
 \label{f:Lshape}
\end{figure}

\subsection{Conclusions}\label{ss:conclusions}

The considered methods allow for comparison in one direction
\begin{align*}
 \|\nabla(u-\up)\|+\|p-\pp\|
 &\lesssim \|\nabla(u-\ubr)\|+\|p-\pbr\|\\
 &\lesssim \|\gradnc(u-\ucr)\|+\|p-\pcr\|\\
 &\lesssim \|\nabla(u-\um)\|+\|p-\pmini\|+\|h_\tri f\|.
\end{align*}
In typical situations, for example, if $p \not\equiv 0$, numerical experiments 
show, that 
all the
methods (and also stabilised $P_1P_0$ , discontinuous Galerkin, and pseudostress
FEMs) exhibit equivalent accuracy on a per-degree-of-freedom basis.

It is clear that this observation disregards other measures for the quality and 
performance of FEMs such as application-driven error functionals or even the 
effort of implementation. Other performance measures may lead to different 
assessments of the methods in practice.

\section{Acknowledgements}
This work is supported
by the DFG Research Center 
Matheon Berlin and by the WCU program 
through KOSEF (R31-2008-000-10049-0).
M.~Schedensack is supported by the Berlin Mathematical School.

\bibliographystyle{plain}
\bibliography{comparisonStokes}

\end{document}